\documentclass{article}
\usepackage[centering,total={410pt,610pt}]{geometry}
\usepackage{graphicx}

\usepackage{tikz}
\usetikzlibrary{calc}
\usepackage{pgfplots}
\pgfplotsset{compat=newest}

\usepackage{enumerate}

\usepackage{amsmath}
\usepackage{amssymb}
\usepackage{amsthm}
\newtheorem{theorem}{Theorem}[section]

\theoremstyle{definition}

\theoremstyle{remark}

\DeclareMathOperator{\tr}{Tr}

\usepackage{float}
\usepackage{hyperref}
\hypersetup{colorlinks,citecolor=black,filecolor=black,linkcolor=black,urlcolor=black}
\floatstyle{ruled}
\newfloat{algorithm}{htbp}{loa}
\floatname{algorithm}{\smallskip Algorithm\smallskip}

\numberwithin{equation}{section}

\usepackage{color}

\title{On the Optimal Ergodic Sublinear Convergence Rate of the Relaxed Proximal Point Algorithm for Variational Inequalities}
\author{%
Guoyong Gu%
\thanks{Department of Mathematics, Nanjing University, Nanjing, 210093, China.}
\thanks{Email: ggu@nju.edu.cn. This author was supported by the NSFC grant 11671195.}
\and
Junfeng Yang\footnotemark[1]
\thanks{Email: jfyang@nju.edu.cn. This author was supported by the NSFC grant  11771208.}
}
\date{}

\begin{document}
\maketitle

\begin{abstract}
This paper investigates the optimal ergodic sublinear convergence rate of the relaxed proximal point algorithm for solving monotone variational inequality problems. The exact worst case convergence rate is computed using the performance estimation framework. It is observed that, as the number of iterations getting larger, this numerical rate asymptotically coincides with an existing sublinear rate, whose optimality is unknown. This hints that, without further assumptions, sublinear convergence rate is likely the best achievable rate. A concrete example is constructed, which provides a lower bound for the exact worst case convergence rate. Amazingly, this lower bound coincides with the exact worst case bound computed via the performance estimation framework. This observation motivates us to conjecture that the lower bound provided by the example is exactly the worse case iteration bound, which is then verified theoretically. We thus have established an ergodic sublinear convergence rate that is optimal in terms of both the order of the sublinear rate and  all the constants involved.


\bigskip

\end{abstract}

\section{Introduction}
\label{Sec:Introduction}

The framework of \emph{Proximal Point Algorithm} (PPA), which dated back to \cite{Mor65bsmf} and was firstly introduced to the optimization community in \cite{Mar70}, has been playing fundamental roles both theoretically and algorithmically in the optimization area, see, e.g., ~\cite{Roc76sicon,Gul91sicon,Gul92siopt,YF00siopt, GHY14coa} for a few of seminal works.
Even the convergence rate of the famous alternating direction method of multipliers is sometimes established within the framework of PPA, see, e.g., ~\cite{HY12sinum}.

Drori and Teboulle \cite{DT14mp} were the first to consider the notion of performance estimation problem.
The performance estimation problem is expressed as a non-convex quadratic matrix problem \cite{Bec06siopt}, which is further relaxed and dualized.
As a consequence, only upper bounds on the worst case performance can be provided.
Subsequently, Kim and Fessler~\cite{KF16MPA,KF17JOTA,KF18SIOPTa,KF18SIOPTb,KF18arxiv} used the idea of  performance estimation problem to establish analytically tighter or optimal upper bounds for either existing or newly proposed accelerated gradient methods for smooth or composite convex minimization problems. 
Later, by using the concept of convex interpolation and smooth (strongly) convex interpolation, the performance estimation problem is transformed into a semidefinite optimization problem without any relaxation by Taylor et al.~\cite{THG17mp,THG17siopt,dKGT17ol,THG18jota}.  See also recent progress on the proximal point method for maximal monotone operator inclusion problems by using the performance estimation framework \cite{GY19nonergodic,Kim19}.

This paper investigates the optimal ergodic sublinear convergence rate of the relaxed PPA for monotone variational inequality problems. Our analysis is motivated by the coincidence of the exact worst case convergence rate computed via the performance estimation framework and the one reached by a concrete example.
The rest of the paper is organized as follows.
In the next section, we briefly review the conceptual relaxed PPA for solving monotone variational inequality problems.
The performance estimation framework is used to compute the exact worst case convergence rate in Section~3, which is followed by a concrete example providing lower bound for the exact worst case convergence rate in Section~4.
The optimal ergodic sublinear convergence rate is established in Section~5
and, finally, some concluding remarks are given in Section~6.

\section{Conceptual relaxed PPA}

Denote the set of monotone and continuous
mappings on $\mathbb{R}^n$ by $\mathcal{M}$.
Let $F\in\mathcal{M}$ and $\Omega \subseteq \mathbb{R}^n$ be a nonempty closed convex set.
We consider the monotone variational inequality problem $\mathrm{VI}(F,\Omega)$: find $w^* \in \Omega$ such that
\begin{equation}
\label{VI}
\tag{VI$(F,\Omega)$}
(w-w^*)^T F(w^*) \geq 0,  \quad \forall w\in\Omega.
\end{equation}

Let $H\in \mathbb{S}^{n}_{++}$ be a real symmetric positive definite matrix of order $n$ and $\lambda \in (0,2)$ be a parameter.
One iteration of our conceptual \emph{Relaxed PPA} (RPPA), which generates $w_{k+1}$ from $w_k$ via the intermediate point $\tilde{w}_{k}$, is given in Algorithm~\ref{Alg:RPPA},
which is a special case of the one considered in \cite{GHY14coa}.
\begin{algorithm}
\caption{An iteration of the conceptual  RPPA for solving VI$(F,\Omega)$}
\label{Alg:RPPA}
\smallskip

\textbf{PPA step:} generate $\tilde{w}_k \in\Omega$ via solving
\begin{equation}
\label{RPPA}
(w-\tilde{w}_{k})^T\big[F(\tilde{w}_{k}) + H(\tilde{w}_{k}-w_k)\big] \geq 0,\
\forall w \in \Omega.
\end{equation}
\textbf{Relaxation step:} generate the new iterate $w_{k+1}$ via
\begin{equation}
\label{RStep}
w_{k+1}= w_k+\lambda(\tilde{w}_k-w_k).
\end{equation}
\end{algorithm}

\noindent Since $F$ is continuous, VI$(F,\Omega)$ is equivalent to find $w^*\in\Omega$ such that
\[
(w-w^*)^T F(w) \geq 0,  \quad \forall w\in\Omega,
\]
see, e.g., \cite{FacP-book-vol1,Nest09MPB}.
Based on this fact, $O(1/N)$ ergodic rate of convergence have been established in \cite{HY12sinum, GHY14coa}, where $N$ denotes the iteration number.
Here the sublinear convergence rate is recalled without a proof.
\begin{theorem}
\label{complexthem}
Let $w_0\in\mathbb{R}^n$ be any initial point, $H\in \mathbb{S}^{n}_{++}$ and $\lambda \in (0,2)$. For all $k\geq 0$,  let $\tilde{w}_k$ be generated by the PPA step \eqref{RPPA}
and $w_{k+1}$ be given by the relaxation step \eqref{RStep}.
Then, for any integer $N>0$,  $\bar{w}_N={1\over N+1}\sum_{k=0}^N\tilde{w}_k$ satisfies
\begin{equation*}
(\bar{w}_N-w)^TF(w) \leq{1\over 2\lambda(N+1)}\|w-w_0\|_H^2,\quad \forall w\in\Omega.
\end{equation*}
\end{theorem}

\noindent For ease of reference, we refer to the factor ${1\over 2\lambda(N+1)}$ in the above sublinear bound as the performance measure factor of RPPA.

\section{Performance estimation framework}
Consider the class of problems ${\cal P}$ given by
\[
{\cal P} = \{(F,\Omega) \mid F\in{\cal M}, \; \emptyset \neq \Omega\subseteq\mathbb{R}^n \text{~is closed and convex}
\}.
\]
%
%
Fix $N>0$ and $\lambda \in (0,2)$. Given $(F,\Omega)\in{\cal P}$ and an initial point $w_0\in\mathbb{R}^n$,  RPPA generates a unique sequence of points $\{\tilde{w}_0, w_1, \tilde{w}_1, w_2, \ldots, \tilde{w}_N\}$, the procedure of which is denoted by
\begin{equation}\label{gPPA}
   \{\tilde{w}_0, w_1, \tilde{w}_1, w_2, \ldots, \tilde{w}_N\} = \text{RPPA}(F,\Omega,w_0,N).
\end{equation}
In this paper, we are interested in establishing the exact worst case complexity bound of  RPPA for the whole class of problems ${\cal P}$.
For such purpose, we consider the following performance estimation problem
\begin{eqnarray}\label{PEP0}
\varepsilon^*(N) \triangleq
  \sup_{F, \Omega, w, w_0}
  \left\{\varepsilon(F,\Omega,w,w_0) \triangleq \frac{(\bar{w}_N-w)^TF(w)}{\|w-w_0\|_H^2}
  \left|
  \begin{array}{l}
    (F,\Omega) \in {\cal P}, \;     w_0\in\mathbb{R}^n, \; w\in\Omega,  \medskip \\
      \bar{w}_N={1\over N+1}\sum_{k=0}^N\tilde{w}_k, \text{~where~} \medskip\\
     \{\tilde{w}_0,    \ldots, \tilde{w}_N\} = \text{RPPA}(F,\Omega,w_0,N).
  \end{array}
  \right.
  \right\}
\end{eqnarray}
with $F$, $\Omega$, $w$ and $w_0$ all being the optimization variables.
It follows from Theorem \ref{complexthem} that  $\varepsilon^*(N) \leq {1\over 2\lambda(N+1)}$ for all $N>0$.

\subsection{Homogeneity}
For any $R>0$, we define $G(u)=F(Ru)/R$ and $\Omega' = \Omega/R$. Clearly, $(F,\Omega)\in \mathcal{P}$ if and only if $(G,\Omega')\in \mathcal{P}$.
Furthermore, it is easy to verify that \eqref{gPPA} holds if and only if
\[\{\tilde{w}_0/R, w_1/R, \tilde{w}_1/R, w_2/R, \ldots, \tilde{w}_N/R\} = \text{RPPA}(G,\Omega',w_0/R,N).\]
It is thus without loss of generality to assume $\|w-w_0\|_H=1$ in  \eqref{PEP0} since
\begin{eqnarray*}
\varepsilon(F,\Omega,w,w_0) =  \frac{(\bar{w}_N-w)^TF(w)}{\|w-w_0\|_H^2} = \frac{(\bar{w}_N/R-w/R)^TG(w/R)}{\|w/R-w_0/R\|_H^2}
= \varepsilon(G,\Omega',w/R,w_0/R).
\end{eqnarray*}

%
%
%

\subsection{Translation invariance}
For any $w_0\in\mathbb{R}^n$, we let $G(u) = F(u+w_0)$ and $\Omega' = \Omega-w_0 = \{w - w_0: w\in\Omega\}$. It is apparent that $(F,\Omega)\in \mathcal{P}$ if and only if $(G,\Omega')\in \mathcal{P}$.
Furthermore, it is elementary to show that \eqref{gPPA} holds if and only if
\[\{\tilde{w}_0-w_0, w_1-w_0, \tilde{w}_1-w_0, w_2-w_0, \ldots, \tilde{w}_N - w_0\} = \text{RPPA}(G,\Omega',0,N).\]
Therefore, we can always choose $w_0 = 0$  in  \eqref{PEP0} without affecting the worst case bound  since
\begin{eqnarray*}
\varepsilon(F,\Omega,w,w_0) =  \frac{(\bar{w}_N-w)^TF(w)}{\|w-w_0\|_H^2} =
\frac{((\bar{w}_N-w_0)-(w-w_0))^TG(w-w_0)}{\|(w -w_0)-0\|_H^2}
= \varepsilon(G,\Omega',w-w_0,0).
\end{eqnarray*}


\subsection{The performance estimation problem}
In view of the homogeneity and translation invariance of the performance measure \eqref{PEP0},
the worst case complexity bound of the RPPA can now be reformulated as
\begin{align}
\notag
\varepsilon^*(N) =
\sup_{F, \ \Omega, \ w,\ w_0}\quad &(\bar{w}_N-w)^TF(w)=\Bigl({1\over N+1}\sum_{k=0}^N\tilde{w}_k-w\Bigr)^TF(w)\\
\tag{PEP}
\text{s.t.}\quad &(F,\Omega)\in \mathcal{P},\ \|w\|_H=1,  w\in\Omega, \\
\notag
& \{\tilde{w}_0, w_1, \tilde{w}_1, w_2, \ldots, \tilde{w}_N\} = \text{RPPA}(F,\Omega, w_0=0,N).
\end{align}
Let $N_{\Omega}(w)$ be the cone normal to $\Omega$ at $w$. Then, the first $(N+1)$ iterations of the RPPA given in Algorithm~\ref{Alg:RPPA} can be rephrased as
\begin{align}
\label{ocrppa}
H(w_k - \tilde{w}_k) \in  (F+N_{\Omega})(\tilde{w}_k),  \quad w_{k+1} = w_k+\lambda(\tilde{w}_k-w_k),\quad k=0, 1,\ldots, N.
\end{align}
Since $\Omega$ is nonempty closed and convex, it is well known that $N_{\Omega}$ is monotone (in fact, maximally monotone).
%
%
Therefore, it follows from the monotonicity of $F$ that $F+N_{\Omega}$ is also monotone.
An optimal solution to (PEP) corresponds to a monotone operator $F + N_{\Omega}$,
on which the RPPA behaves as badly as possible.
Because it involves an unknown monotone operator $F + N_{\Omega}\in \mathcal{M}$,
problem (PEP) is  infinite-dimensional.
Following the ideas in \cite{THG17siopt,THG17mp,DT16mp,RTBG18},
we will show that a completely equivalent finite-dimensional convex semidefinite optimization problem can readily be formulated.

For any set valued mapping $T$, we denote its graph by $\text{graph}(T) = \{(x,y): y\in Tx\}$.
Then, the inclusion conditions in \eqref{ocrppa} and the constraint $w\in\Omega$ can be restated as
\begin{align}\label{inclusion_conditions}
\notag
(\tilde{w}_k, H(w_k - \tilde{w}_k)) & \in  \text{graph}(F+N_{\Omega}), \quad k=0, 1,\ldots, N, \\
(w,F(w)) & \in \text{graph}(F+N_{\Omega}).
\end{align}
Since $(F+N_\Omega)$ is monotone, by definition the set of inclusions in \eqref{inclusion_conditions} imply
\begin{align}\label{interpolation_conditions}
\notag
\bigl\langle \tilde{w}_{j}-\tilde{w}_{i},\ H(w_j - \tilde{w}_j)-H(w_i - \tilde{w}_i)\bigr\rangle&\geq0,
\quad 0\leq i<j\leq N,\\
\bigl\langle \tilde{w}_i-w,\ H(w_i - \tilde{w}_i)-F(w)\bigr\rangle&\geq0,
\quad 0\leq i\leq N.
\end{align}
On the other hand, if the set of points $S := \bigl\{(w, F(w)),\ \bigl(\tilde{w}_k,\ H(w_k - \tilde{w}_k)\bigr)_{k=0, 1,\ldots, N}\bigr\}$ satisfies the set of interpolation conditions \eqref{interpolation_conditions}, then the operator  $T_S(x) := \{g \mid (x,g) \in S\}$ is monotone and satisfies
\begin{align*}
\notag
(\tilde{w}_k, H(w_k - \tilde{w}_k)) & \in  \text{graph}(T_S), \quad k=0, 1,\ldots, N, \\
(w,F(w)) & \in \text{graph}(T_S).
\end{align*}
Since $\tilde{w}_k$ is uniquely determined by $w_k$, it follows that $T_S$ is single-valued. By defining $\tilde{F}(w) := T_S(w)$, $\tilde{F}(\tilde{w}_k) := T_S(\tilde{w}_k)$ and letting $F$ be a single valued monotone extension of $\tilde{F}$, then
the set of inclusion conditions given in \eqref{inclusion_conditions} is satisfied with $\Omega$ being the convex hull of $\{\tilde{w}_k: k=0,1,\ldots,N \} \cup \{w\}$, see {\cite{Taylor17,BC17book}}.

Therefore, the infinite-dimensional constraints $(F,\Omega)\in\mathcal{P}$ and $w\in\Omega$ can be replaced with the set of inequalities \eqref{interpolation_conditions}.
As long as the dimension of $F$ is larger than or equal to $N+3$ (the dimension of the Gram matrix in the subsequent subsection), the replacement is just a reformulation while not a relaxation, which means that the exact worst-case performance measure factor $\varepsilon^*(N)$ can be computed.

\subsection{Gram representation}
In order to reduce the number of variables, we substitute $w_k$, for $k=1, 2, \ldots, N$, using relaxation step (\ref{RStep}) and $w_0=0$, namely,
\begin{align}
\notag
w_1 &=\lambda \tilde{w}_0,\\
\notag
w_2 &=(1-\lambda)\lambda \tilde{w}_0+\lambda\tilde{w}_1,\\
\notag
&\cdots\cdots\\
\label{wk}
w_k &=(1-\lambda)^{k-1}\lambda\tilde{w}_0+\cdots+(1-\lambda)\lambda\tilde{w}_{k-2}+\lambda\tilde{w}_{k-1},\\
\notag
&\cdots\cdots
\end{align}
To obtain a convex formulation for (PEP), we introduce a Gram matrix to describe the iterates and their mappings. Denoting
$$P=\bigl[H^{1/2}\tilde{w}_0, H^{1/2}\tilde{w}_1, \cdots, H^{1/2}\tilde{w}_N,H^{1/2}w, H^{-1/2}F(w)\bigr],$$
we define the symmetric $(N+3)\times (N+3)$ Gram matrix $G=P^TP\in \mathbb{S}^{N+3}$.
The constraints and the objective function in problem (PEP) can now be entirely formulated in terms of the entries of the Gram matrix $G$.

By combining (\ref{ocrppa}) and (\ref{wk}), we have
$$H\bigl[(1-\lambda)^{k-1}\lambda\tilde{w}_0+\cdots+(1-\lambda)\lambda\tilde{w}_{k-2}+\lambda\tilde{w}_{k-1}-\tilde{w}_k\bigr]\in (F+N_{\Omega})(\tilde{w}_k).$$
Thus, for $0\leq i<j\leq N$,
\begin{align*}
&\Bigl\langle \tilde{w}_j-\tilde{w}_i,\ H(w_j - \tilde{w}_j)-H(w_i - \tilde{w}_i) \Bigr\rangle\\
&=\Bigl\langle \tilde{w}_j-\tilde{w}_i,\ H\bigl[(1-\lambda)^{j-1}\lambda\tilde{w}_0+\cdots+\lambda\tilde{w}_{j-1}-\tilde{w}_j\bigr]
-H\bigl[(1-\lambda)^{i-1}\lambda\tilde{w}_0+\cdots+\lambda\tilde{w}_{i-1}-\tilde{w}_i\bigr]\Bigr\rangle\\
&=\tr\Bigl(G\Bigl\{\bigl[(1-\lambda)^{j-1}\lambda e_1+\cdots+ \lambda e_j-e_{j+1}\bigr]\\
&\qquad\qquad\qquad\qquad\qquad\qquad-\bigl[(1-\lambda)^{i-1}\lambda e_1+\cdots+ \lambda e_i-e_{i+1}\bigr]\Bigr\}(e_{j+1}-e_{i+1})^T\Bigr)\\
&\triangleq \tr(GA_{ij}).
\end{align*}
In the same vein, for $i=0, \ldots, N$,
\begin{align*}
&\Bigl\langle \tilde{w}_i-w,\ H(w_i - \tilde{w}_i)-F(w) \Bigr\rangle\\
&=\Bigl\langle \tilde{w}_i-w,\ H\bigl[(1-\lambda)^{i-1}\lambda\tilde{w}_0+\cdots+\lambda\tilde{w}_{i-1}-\tilde{w}_i\bigr]-F(w) \rangle\\
&=\tr\Bigl(G [(1-\lambda)^{i-1}\lambda e_1+\cdots+ \lambda e_i-e_{i+1}-e_{N+3}](e_{i+1}-e_{N+2})^T\Bigr)\\
&\triangleq \tr(GA_{iw}),
\end{align*}
and
\begin{align*}
\|w\|_H^2=\tr\Bigl(Ge_{N+2}e_{N+2}^T\Bigr)\triangleq \tr(GA_w).
\end{align*}
Finally, the objective function
\begin{align*}
\max\ (\bar{w}_N-w)^TF(w)
&=-\min\ (w-\bar{w}_N)^TF(w)\\
&=-\min\  \tr\Bigl(G \bigl(e_{N+2}-\frac{1}{N+1}(e_1+\cdots+e_{N+1})\bigr)e_{N+3}^T\Bigr)\\
&\triangleq -\min\  \tr(GC).
\end{align*}

\subsection{Numerical results}
We arrive at the following equivalent semidefinite optimization problem of (PEP):
\begin{align*}
\varepsilon^*(N) =
- \min\quad & \tr\bigl(G \cdot (C+C^T)/2\bigr)\\
\text{s.t.}\quad & \tr\bigl(G\cdot (A_{ij}+A_{ij}^T)/2\bigr)\geq 0, \quad 0\leq i<j\leq N,\\
& \tr\bigl(G\cdot (A_{iw}+A_{iw}^T)/2\bigr)\geq 0, \quad 0\leq i\leq N,\\
& \tr\bigl(G\cdot (A_{w}+A_{w}^T)/2\bigr)=1,\\
& G\succeq 0.
\end{align*}
We use SeDuMi \cite{Stu99oms} to solve the above semidefinite optimization problem for $N=1, 2, \ldots, 50$.
In terms of performance measure factor, the worst case iteration bound just computed is compared to the existing iteration bound given in Theorem~\ref{complexthem} for $\lambda=1.5$,
which is depicted in Figure~\ref{figpep}.
We observe that the existing iteration bound coincides with the numerical bound asymptotically, when the number of iterations $N$ gets larger.
This could hint that, without further assumptions, the existing sublinear rate ${1\over 2\lambda(N+1)}$ is likely the best achievable rate in terms of the order of the sublinear rate. In the next section, we establish a lower bound on $\varepsilon^*(N)$ via constructing a concrete example.
Although this lower bound is only slightly better than ${1\over 2\lambda(N+1)}$, it is optimal.


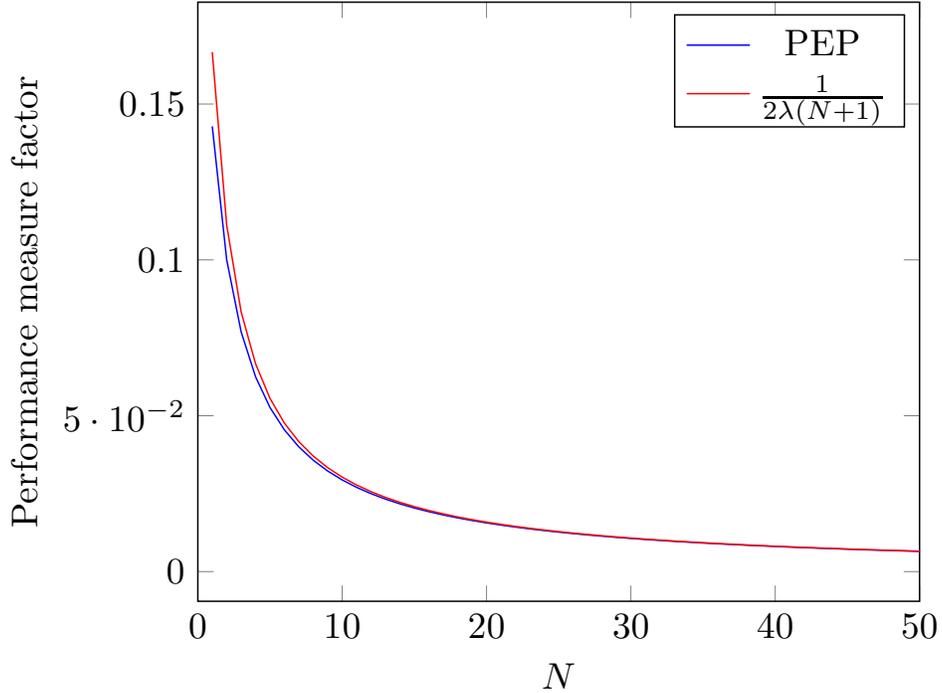
\begin{figure}
\centering\begin{tikzpicture}[scale=1.4]
\begin{axis}[xmin=0, xmax=50, xlabel={$N$}, ylabel={Performance measure factor}]
\addplot[mark=, blue] table[x=N,y=f1]{result50.dat};
\addlegendentry{PEP}
\addplot[mark=, red] table[x=N,y=f2]{result50.dat};
\addlegendentry{$\frac{1}{2\lambda(N+1)}$}
\end{axis}
\end{tikzpicture}
\caption{Comparison between the worst case bound computed by PEP (the lower curve) and the existing theoretical bound (the upper curve).
}
\label{figpep}
\end{figure}

\section{A lower bound for $\varepsilon^*(N)$}\label{sc:example}
It follows from Theorem \ref{complexthem} that the exact worst case convergence rate of RPPA, i.e., $\varepsilon^*(N)$ defined in \eqref{PEP0}, is bounded above by ${1\over 2\lambda(N+1)}$.
In this section, we construct a concrete one-dimensional example to show that $\varepsilon^*(N) \geq {1\over 2(\lambda N + 2)}$.


Let $\Omega=\mathbb{R}$, $H = 1$, $c>0$, $\delta>0$ and define
\begin{equation}\label{F}
F(w)=\begin{cases}
c, & w \in (\delta, +\infty),\\
{c \over \delta} w, & w\in [-\delta,\delta],\\
-c, & w \in (-\infty, -\delta),
\end{cases}
\end{equation}
which is clearly single-valued, monotone and continuous.
Let $N>0$ be a fixed integer and $w_0 \in (\delta, +\infty)$ be an initial point.
The idea is to choose the parameter $c>0$ sufficiently small such that $\tilde{w}_k > \delta$ for $k=0, 1, \ldots, N$.
According to Algorithm~\ref{Alg:RPPA} and the definition of $F$ in \eqref{F}, it is elementary to show that
\begin{alignat*}{3}
&c+\tilde{w}_0-w_0=0,\qquad & \tilde{w}_0&= w_0-c,\qquad  & w_1&=w_0+\lambda(\tilde{w}_0-w_0)=w_0-\lambda c,\\
&c+\tilde{w}_1-w_1=0,\qquad & \tilde{w}_1&= w_1-c,\qquad  & w_2&=w_1+\lambda(\tilde{w}_1-w_1)=w_0-2\lambda c,\\
&\quad\cdots\cdots & &\quad\cdots\cdots & &\quad\cdots\cdots \\
&c+\tilde{w}_N-w_N=0, & \tilde{w}_N&= w_N-c,  & w_{N+1}&=w_N+\lambda(\tilde{w}_N-w_N)=w_0-(N+1)\lambda c.
\end{alignat*}
Thus, we have
$$\tilde{w}_k= w_k-c = w_0-c-k\lambda c,\quad k=0, 1, \ldots, N.$$
We choose $c$ such that
\begin{equation}
\label{c}
0 < c < \frac{w_0 - \delta}{\lambda N+ 2},
\end{equation}
which guarantees $\tilde{w}_k > \delta$ for $k=1, 2, \ldots, N$.
Consequently, we have
$$\bar{w}_N=\frac{1}{N+1}(\tilde{w}_0+\tilde{w}_1+\cdots+\tilde{w}_N)
=w_0-\frac{\lambda N+2}{2}c.$$
Then, the exact worst complexity bound $\varepsilon^*(N)$ satisfies
\begin{align}\label{ex:case-0}
\varepsilon^*(N) \geq
\max_{w\neq w_0}\ \frac{\bigl\langle \bar{w}_N-w,\ F(w)\bigr\rangle}{\|w-w_0\|^2}.
\end{align}
We next separate the right hand side of \eqref{ex:case-0} into three intervals: (i) $w>\delta$, (ii) $w\in [-\delta,\delta]$, and (iii) $w<-\delta$.

Case (i): $w>\delta$. In this case, $F(w) = c$ and thus
\begin{align}\label{ex:case-1}
\max_{w_0\neq w>\delta}\ \frac{\bigl\langle \bar{w}_N-w,\ F(w)\bigr\rangle}{\|w-w_0\|^2}
=\max_{w_0\neq w>\delta} \frac{\Bigl(w_0-w-\frac{\lambda N+2}{2}c\Bigr)c}{(w_0-w)^2}.
\end{align}
When $w_0-w-\frac{\lambda N+2}{2}c<0$, i.e., $w_0-w<\frac{\lambda N+2}{2}c$, the objective is negative.
Thus, we only need to consider the case $w_0-w\geq \frac{\lambda N+2}{2}c$. Then, \eqref{ex:case-1} may be reformulated as
$$\max_{w_0\neq w>\delta}\ \Bigl[\frac{c}{w_0-w}-\frac{\lambda N+2}{2}\frac{c^2}{(w_0-w)^2}\Bigr].$$
This is a quadratic function in terms of $\frac{1}{w_0-w}$, and the maximum is attained at $\frac{1}{w_0-w}=\frac{1}{(\lambda N+2)c}$, since, in view of \eqref{c}, $w=w_0-(\lambda N+2)c > \delta$.
The maximum objective value is $\frac{1}{2(\lambda N+2)}$.
%
%

Case (ii): $w\in [-\delta,\delta]$. In this case, we have $F(w) = {c \over \delta} w$ and
\begin{equation}\label{ex:case-2}
  w_0-w-\frac{\lambda N+2}{2}c > \frac{\lambda N+2}{2}c + \delta - w \geq \frac{\lambda N+2}{2}c > 0,
\end{equation}
where the first ``$>$" follows from \eqref{c}, while the second inequality is because $w \leq \delta$.
Thus,
\begin{align*}
\max_{w\in [-\delta,\delta]}\ \frac{\bigl\langle \bar{w}_N-w,\ F(w)\bigr\rangle}{\|w-w_0\|^2}
&=\max_{w\in [-\delta,\delta]}\ \frac{\Bigl(w_0-w-\frac{\lambda N+2}{2}c\Bigr)  {c\over \delta} w }{(w_0-w)^2}\\
&\leq \max_{w\in [-\delta,\delta]}\ \frac{\Bigl(w_0-w-\frac{\lambda N+2}{2}c\Bigr)  c }{(w_0-w)^2} \\
&\leq \max_{w\neq w_0}\ \frac{\Bigl(w_0-w-\frac{\lambda N+2}{2}c\Bigr)  c }{(w_0-w)^2} \\
& = \max_{w\neq w_0}\ \Bigl[\frac{c}{w_0-w}-\frac{\lambda N+2}{2}\frac{c^2}{(w_0-w)^2}\Bigr] \\
& = \frac{1}{2(\lambda N+2)},
\end{align*}
where the first ``$\leq$" follows from \eqref{ex:case-2} and $w\leq \delta$, and the last ``$=$" can be deduced by following the same arguments about Case (i).

Case (iii): $w< -\delta$. In this case, $F(w)=-c$ and thus
\begin{align}\label{ex:case-3}
\max_{w<-\delta}\ \frac{\bigl\langle \bar{w}_N-w,\ F(w)\bigr\rangle}{\|w-w_0\|^2}
=\max_{w<-\delta}\ \frac{\frac{\lambda N+2}{2}c-(w_0-w)}{(w_0-w)^2}c.
\end{align}
In view of (\ref{c}) and $w<-\delta$, we derive
\begin{align*}
\frac{\lambda N+2}{2}c-(w_0-w) < \frac{\lambda N+2}{2}c-w_0
<  -\frac{\lambda N+2}{2}c - \delta <0,
\end{align*}
which means that \eqref{ex:case-3} is negative.

In summary, for $F$ defined in \eqref{F} and $\Omega=\mathbb{R}$, if $w_0 > \delta >0$ and $c\in(0, \frac{w_0 - \delta}{\lambda N+2})$, then $\bar{w}_N$ generated by the RPPA with $H=1$ and any $\lambda \in (0,2)$ satisfies
\begin{equation}\label{ex:final}
(\bar{w}_N-w)^TF(w) \leq{1\over 2(\lambda N+2)}\|w-w_0\|_H^2,\ \forall w\in\mathbb{R},
\end{equation}
where the inequality becomes equality for $w=w_0-(\lambda N+2)c$. For $w=w_0$, \eqref{ex:final} holds trivially because
$(\bar{w}_N-w_0)^TF(w_0) = -{(\lambda N + 2)c \over 2} < 0$.

By combining the result of this example and Theorem \ref{complexthem}, we have shown that for any integer $N>0$ the exact worst case complexity bound $\varepsilon^*(N)$ defined in \eqref{PEP0} satisfies
\[
{1\over 2(\lambda N + 2)} \leq \varepsilon^*(N) \leq {1\over 2\lambda(N  + 1)}.
\]
It is amazing to observe that the performance measure factor of the example constructed, namely ${1\over 2(\lambda N+2)}$,
coincides with the factor computed by the performance estimation framework,
with the largest difference being $4.62\times 10^{-8}$, for $N=1, 2,  3, \ldots, 100$.
This hints that the exact worst case iteration bound is likely attained on the example we just constructed.

\section{Optimal ergodic sublinear convergence rate}
An upper bound on the optimal ergodic sublinear convergence rate
is established in this section. Since this upper bound matches exactly the lower bound provided by the example in Section 4,
we have thus established the optimal ergodic sublinear convergence rate.

\begin{theorem}
\label{complexthemoptimal}
Let $w_0\in\mathbb{R}^n$ be any initial point, $H\in \mathbb{S}^{n}_{++}$ and $\lambda \in (0,2)$. For all $k\geq 0$,  let $\tilde{w}_k$ be generated by the PPA step \eqref{RPPA}
and $w_{k+1}$ be given by the relaxation step \eqref{RStep}.
Then, for any integer $N>0$,  $\bar{w}_N={1\over N+1}\sum_{k=0}^N\tilde{w}_k$ satisfies
\begin{equation*}
(\bar{w}_N-w)^TF(w) \leq{1\over 2(\lambda N+2)}\|w-w_0\|_H^2,\ \forall w\in\Omega.
\end{equation*}
\end{theorem}
\begin{proof}
Let $k\in\{0, 1, \ldots, N\}$ and $w\in\Omega$ be arbitrarily fixed.
The monotonicity of $F$ implies that
$$\bigl\langle \tilde{w}_k-w,\ F(w) \bigr\rangle\leq \bigl\langle \tilde{w}_k-w,\ F(\tilde{w}_k)\bigr\rangle.$$
Combing (\ref{RPPA}) with the above inequality yields
\begin{align}
\notag
\bigl\langle \tilde{w}_k-w,\ F(w) \bigr\rangle
&\leq \bigl\langle \tilde{w}_k-w,\ F(\tilde{w}_k)\bigr\rangle
\leq (w-\tilde{w}_k)^TH(\tilde{w}_k-w_k)\\
\label{wwkHtwwk}
&=(w-w_k)^TH(\tilde{w}_k-w_k)-\|\tilde{w}_k-w_k\|_H^2.
\end{align}
The relaxation step (\ref{RStep}) is equivalent to $\tilde{w}_k-w_k={1\over\lambda}(w_{k+1}-w_k)$.
Substitution into inequality (\ref{wwkHtwwk}) yields
\begin{align*}
\bigl\langle \tilde{w}_k-w,\ F(w) \bigr\rangle
&\leq \frac{1}{\lambda}(w-w_k)^TH(w_{k+1}-w_k)-\frac{1}{\lambda^2}\|w_{k+1}-w_k\|_H^2\\
&=\frac{1}{2\lambda}\Bigl[\|w-w_k\|_H^2-\|w-w_{k+1}\|_H^2+\|w_{k+1}-w_k\|_H^2\Bigr]-\frac{1}{\lambda^2}\|w_{k+1}-w_k\|_H^2.
\end{align*}
Summing the above inequality over $k=0, 1, \ldots, N$, we obtain that
\begin{align}
\notag
&(N+1)\Bigl\langle {1\over N+1} \sum_{k=0}^N \tilde{w}_k-w,\ F(w)\Bigr\rangle\\
\notag
&\leq \frac{1}{2\lambda}\Bigl[\|w-w_0\|_H^2-\|w-w_{N+1}\|_H^2\Bigr]-\frac{2-\lambda}{2\lambda^2}\sum_{k=0}^N\|w_{k+1}-w_k\|_H^2\\
\notag
&\leq \frac{1}{2\lambda}\Bigl[\|w-w_0\|_H^2-\|w-w_{N+1}\|_H^2\Bigr]-\frac{2-\lambda}{2\lambda^2}\frac{\Bigl(\sum_{k=0}^N\|w_{k+1}-w_k\|_H\Bigr)^2}{N}\\
\notag
&\leq \frac{1}{2\lambda}\Bigl[\|w-w_0\|_H^2-\|w-w_{N+1}\|_H^2\Bigr]-\frac{2-\lambda}{2\lambda N^2}\|w_{N+1}-w_0\|_H^2\\
\label{Np1}
&=\frac{1}{2\lambda}\|w-w_0\|_H^2-\frac{1}{2\lambda}\|w-w_0\|_H^2
\biggl[\frac{\|w-w_{N+1}\|_H^2}{\|w-w_0\|_H^2}+\frac{2-\lambda}{\lambda N}\frac{\|w_{N+1}-w_0\|_H^2}{\|w-w_0\|_H^2}\biggr],
\end{align}
where the 2nd and the 3rd ``$\leq$" follow from the Cauchy-Schwarz and the triangle inequalities, respectively.
To simplify the notation, we denote
$$A:=\frac{\|w-w_{N+1}\|_H}{\|w-w_0\|_H}\geq0,\quad B:=\frac{\|w_{N+1}-w_0\|_H}{\|w-w_0\|_H}\geq 0,$$
then
\begin{align*}
\frac{\|w-w_{N+1}\|_H^2}{\|w-w_0\|_H^2}+\frac{2-\lambda}{\lambda N}\frac{\|w_{N+1}-w_0\|_H^2}{\|w-w_0\|_H^2}
=A^2+\frac{2-\lambda}{\lambda N}B^2.
\end{align*}
Observing that
$$A+B=\frac{\|w-w_{N+1}\|_H}{\|w-w_0\|_H}+\frac{\|w_{N+1}-w_0\|_H}{\|w-w_0\|_H}\geq
1.$$
If $B\geq 1$, we have
\begin{equation}
\label{Bgeq1}
A^2+\frac{2-\lambda}{\lambda N}B^2\geq \frac{2-\lambda}{\lambda N}\geq \frac{2-\lambda}{\lambda N+2}.
\end{equation}
If $0\leq B<1$, then $A+B\geq 1$ implies
\begin{align}
\notag
A^2+\frac{2-\lambda}{\lambda N}B^2
&\geq (1-B)^2+\frac{2-\lambda}{\lambda N}B^2
=\Bigl(1+\frac{2-\lambda}{\lambda N}\Bigr)B^2-2B+1\\
\notag
&=\Bigl(1+\frac{2-\lambda}{\lambda N}\Bigr)\biggl(B-\Bigl(1+\frac{2-\lambda}{\lambda N}\Bigr)^{-1}\biggr)^2+1-\Bigl(1+\frac{2-\lambda}{\lambda N}\Bigr)^{-1}\\
\label{Bleq1}
&\geq \frac{2-\lambda}{\lambda N+2-\lambda}\geq\frac{2-\lambda}{\lambda N+2}.
\end{align}
Combining (\ref{Bgeq1}) and (\ref{Bleq1}), we derive that
\begin{align*}
\frac{\|w-w_{N+1}\|_H^2}{\|w-w_0\|_H^2}+\frac{2-\lambda}{\lambda N}\frac{\|w_{N+1}-w_0\|_H^2}{\|w-w_0\|_H^2}
=A^2+\frac{2-\lambda}{\lambda N}B^2\geq \frac{2-\lambda}{\lambda N+2}.
\end{align*}
Incorporating the last inequality into (\ref{Np1}), we get
\begin{align*}
\notag
(N+1)\Bigl\langle\frac{1}{N+1}\sum_{k=0}^N \tilde{w}_k-w,\ F(w)\Bigr\rangle
&\leq \frac{1}{2\lambda}\|w-w_0\|_H^2 \Bigl(1 -\frac{2-\lambda}{\lambda N+2}\Bigr)\\
&=\frac{N+1}{2(\lambda N +2)}\|w-w_0\|_H^2.
\end{align*}
The convergence result follows by dividing both sides with $N+1$.
\end{proof}
Following the result of this theorem and the example constructed  in Section \ref{sc:example}, we have proved that
\[
 \varepsilon^*(N) = {1\over 2(\lambda N + 2)}.
\]
Note that this complexity bound is optimal in terms of  not only the order of $N$ but also all the constants involved.
In addition, our ideas of establishing the exact worst case complexity bound can be easily extended to the case of mixed variational inequality in \cite{GHY14coa}.

\section{Concluding remarks}
Performance estimation framework is used to compute the exact worst case ergodic sublinear convergence rate of the RPPA for monotone variational inequality problems.
The numerical complexity bound asymptotically coincides with the existing theoretical bound. A concrete example is constructed to provide a lower bound on the worst case sublinear rate.
The coincidence of the numerical bound and the lower bound on the constructed example leads us to conjecture that the lower bound is optimal, which is indeed verified theoretically.
Though the improvement of the bound, namely from $\frac{1}{2\lambda (N+1)}$ to $\frac{1}{2\lambda N+4}$, is minor, kind of asymptotic, and not in the order, it does eliminate the gap, and moreover, it is optimal  in terms of not only the sublinear order but also the constants involved.


\begin{thebibliography}{10}

\bibitem{BC17book}
H.~H. Bauschke and P.~L. Combettes.
\newblock {\em Convex analysis and monotone operator theory in {H}ilbert
  spaces}.
\newblock CMS Books in Mathematics/Ouvrages de Math\'{e}matiques de la SMC.
  Springer, Cham, second edition, 2017.
\newblock With a foreword by H\'{e}dy Attouch.

\bibitem{Bec06siopt}
A.~Beck.
\newblock Quadratic matrix programming.
\newblock {\em SIAM J. Optim.}, 17(4):1224--1238, 2006.

\bibitem{dKGT17ol}
E.~de~Klerk, F.~Glineur, and A.~B. Taylor.
\newblock On the worst-case complexity of the gradient method with exact line
  search for smooth strongly convex functions.
\newblock {\em Optim. Lett.}, 11(7):1185--1199, 2017.

\bibitem{DT14mp}
Y.~Drori and M.~Teboulle.
\newblock Performance of first-order methods for smooth convex minimization: a
  novel approach.
\newblock {\em Math. Program.}, 145(1-2, Ser. A):451--482, 2014.

\bibitem{DT16mp}
Y.~Drori and M.~Teboulle.
\newblock An optimal variant of {K}elley's cutting-plane method.
\newblock {\em Math. Program.}, 160(1-2, Ser. A):321--351, 2016.

\bibitem{FacP-book-vol1}
F.~Facchinei and J.-S. Pang.
\newblock {\em Finite-dimensional variational inequalities and complementarity
  problems. {V}ol. {I}}.
\newblock Springer Series in Operations Research. Springer-Verlag, New York,
  2003.

\bibitem{GHY14coa}
G.~Gu, B.~He, and X.~Yuan.
\newblock Customized proximal point algorithms for linearly constrained convex
  minimization and saddle-point problems: a unified approach.
\newblock {\em Comput. Optim. Appl.}, 59(1-2):135--161, 2014.

\bibitem{GY19nonergodic}
G.~Gu and J.~Yang.
\newblock Optimal nonergodic sublinear convergence rate of proximal point
  algorithm for maximal monotone inclusion problems.
\newblock {\em arXiv:1904.05495}, 2019.

\bibitem{Gul91sicon}
O.~G{\"u}ler.
\newblock On the convergence of the proximal point algorithm for convex
  minimization.
\newblock {\em SIAM J. Control Optim.}, 29(2):403--419, 1991.

\bibitem{Gul92siopt}
O.~G{\"u}ler.
\newblock New proximal point algorithms for convex minimization.
\newblock {\em SIAM J. Optim.}, 2(4):649--664, 1992.

\bibitem{HY12sinum}
B.~He and X.~Yuan.
\newblock On the {$O(1/n)$} convergence rate of the {D}ouglas-{R}achford
  alternating direction method.
\newblock {\em SIAM J. Numer. Anal.}, 50(2):700--709, 2012.

\bibitem{Kim19}
D.~Kim.
\newblock Accelerated proximal point method for maximally monotone operators.
\newblock {\em arXiv:1905.05149}, 2019.

\bibitem{KF16MPA}
D.~Kim and J.~A. Fessler.
\newblock Optimized first-order methods for smooth convex minimization.
\newblock {\em Math. Program.}, 159(1-2, Ser. A):81--107, 2016.

\bibitem{KF17JOTA}
D.~Kim and J.~A. Fessler.
\newblock On the convergence analysis of the optimized gradient method.
\newblock {\em J. Optim. Theory Appl.}, 172(1):187--205, 2017.

\bibitem{KF18SIOPTa}
D.~Kim and J.~A. Fessler.
\newblock Another look at the fast iterative shrinkage/thresholding algorithm
  ({FISTA}).
\newblock {\em SIAM J. Optim.}, 28(1):223--250, 2018.

\bibitem{KF18SIOPTb}
D.~Kim and J.~A. Fessler.
\newblock Generalizing the optimized gradient method for smooth convex
  minimization.
\newblock {\em SIAM J. Optim.}, 28(2):1920--1950, 2018.

\bibitem{KF18arxiv}
D.~Kim and J.~A. Fessler.
\newblock Optimizing the efficiency of first-order methods for decreasing the
  gradient of smooth convex functions.
\newblock {\em arXiv:1803.06600}, 2018.

\bibitem{Mar70}
B.~Martinet.
\newblock R\'egularisation d'in\'equations variationnelles par approximations
  successives.
\newblock {\em Rev. Fran\c caise Informat. Recherche Op\'erationnelle}, 4(Ser.
  R-3):154--158, 1970.

\bibitem{Mor65bsmf}
J.-J. Moreau.
\newblock Proximit\'e et dualit\'e dans un espace hilbertien.
\newblock {\em Bull. Soc. Math. France}, 93:273--299, 1965.

\bibitem{Nest09MPB}
Y.~Nesterov.
\newblock Primal-dual subgradient methods for convex problems.
\newblock {\em Math. Program.}, 120(1, Ser. B):221--259, 2009.

\bibitem{Roc76sicon}
R.~T. Rockafellar.
\newblock Monotone operators and the proximal point algorithm.
\newblock {\em SIAM J. Control Optim.}, 14(5):877--898, 1976.

\bibitem{RTBG18}
E.~K. Ryu, A.~B. Taylor, C.~Bergeling, and P.~Giselsson.
\newblock Operator splitting performance estimation: Tight contraction factors
  and optimal parameter selection.
\newblock {\em arXiv:1812.00146}, 2018.

\bibitem{Stu99oms}
J.~F. Sturm.
\newblock Using {S}e{D}u{M}i 1.02, a {MATLAB} toolbox for optimization over
  symmetric cones.
\newblock {\em Optim. Methods Softw.}, 11/12(1-4):625--653, 1999.
\newblock Interior point methods.

\bibitem{Taylor17}
A.~B. Taylor.
\newblock Convex interpolation and performance estimation of first-order
  methods for convex optimization.
\newblock {\em PhD Thesis, Universite Catholique de Louvain}, 2017.

\bibitem{THG17siopt}
A.~B. Taylor, J.~M. Hendrickx, and F.~Glineur.
\newblock Exact worst-case performance of first-order methods for composite
  convex optimization.
\newblock {\em SIAM J. Optim.}, 27(3):1283--1313, 2017.

\bibitem{THG17mp}
A.~B. Taylor, J.~M. Hendrickx, and F.~Glineur.
\newblock Smooth strongly convex interpolation and exact worst-case performance
  of first-order methods.
\newblock {\em Math. Program.}, 161(1-2, Ser. A):307--345, 2017.

\bibitem{THG18jota}
A.~B. Taylor, J.~M. Hendrickx, and F.~Glineur.
\newblock Exact worst-case convergence rates of the proximal gradient method
  for composite convex minimization.
\newblock {\em J. Optim. Theory Appl.}, 178(2):455--476, 2018.

\bibitem{YF00siopt}
N.~Yamashita and M.~Fukushima.
\newblock The proximal point algorithm with genuine superlinear convergence for
  the monotone complementarity problem.
\newblock {\em SIAM J. Optim.}, 11(2):364--379, 2000.

\end{thebibliography}

\end{document}